\documentclass[11pt,a4paper,oneside]{article}
\normalbaselines 
\usepackage[english]{babel}
\usepackage{graphicx}
\usepackage{amsmath}
\usepackage{amssymb}
\usepackage{mathtools}
\usepackage{mathcomp}
\usepackage{textcomp}
\usepackage{amsthm}
\usepackage{enumerate}
\usepackage{subcaption}
\usepackage[auth-lg]{authblk}
\usepackage[T1]{fontenc}
\newtheorem{theorem}{Theorem}\theoremstyle{plain}
\newtheorem{corollary}{Corollary}\theoremstyle{plain}
\newtheorem{proposition}{Proposition}\theoremstyle{plain}
\newtheorem{lemma}{Lemma}\theoremstyle{plain}
\newtheorem{claim}{Claim}\theoremstyle{plain}
\theoremstyle{plain}
\theoremstyle{plain}
\theoremstyle{plain}

\newcommand{\ignore}[1]{}
\newcommand{\QAP}{\text{\rm QAP}}
\newcommand{\Tr}{\text{\rm Tr}}
\newcommand{\MP}{{\mathcal P}}
\newcommand{\CUT}{\text{\rm CUT}}

\newcommand{\MS}{{\mathcal S}}
\newcommand{\oR}{{\mathbb R}}
\newcommand{\sfT}{{\sf T}}

\begin{document}

\title{\bf The quadratic assignment problem is easy for Robinsonian matrices {with Toeplitz structure}}
\author[1,2]{Monique Laurent}
\author[1]{Matteo Seminaroti}
\date{}

\affil[1]{\small Centrum Wiskunde \& Informatica (CWI), Science Park 123, 1098 XG Amsterdam, The Netherlands}
\affil[2]{\small Tilburg University, P.O. Box 90153, 5000 LE Tilburg, The Netherlands}
\footnotetext{Correspondence to : \texttt{M.Seminaroti@cwi.nl} (M.~Seminaroti), \texttt{M.Laurent@cwi.nl} (M.~Laurent), CWI, Postbus 94079, 1090 GB, Amsterdam. Tel.:+31 (0)20 592 4386.}

\maketitle

\begin{abstract}
We present a new polynomially solvable case of the Quadratic Assignment Problem in Koopmans-Beckman form $\QAP(A,B)$, by showing that the identity permutation is optimal when $A$ and $B$ are respectively a Robinson similarity and dissimilarity matrix and one of $A$ or $B$ is a Toeplitz matrix. 
A Robinson (dis)similarity matrix is a symmetric matrix whose entries (increase) decrease monotonically along rows and columns when moving away from the diagonal, and such matrices arise in the classical seriation problem.\\

\end{abstract} 


\section{Introduction}\label{Introduction}

In this paper we consider two problems over  permutations. Our main problem of interest  is the  {\em Quadratic Assignment Problem} (QAP),    a well studied hard combinatorial optimization  problem, introduced by Koopmans and Beckman \cite{Koopmans57} in 1957 as a mathematical model for 
the location of  indivisible economic activities. In the problem $\QAP(A,B)$, we are given  $n$ facilities, $n$ locations,   a {\em flow matrix} $A$ whose entry $A_{ij}$ represents the flow of activity between two facilities $i$ and $j$,  and a {\em distance matrix} $B$ whose entry $B_{ij}$ represents the distance between the locations $i$ and $j$. Then the objective is to find an assignment of the facilities to the locations, i.e., a permutation $\pi$ of the set $[n]=\{1,\ldots,n\}$, minimizing the total cost of the assignment. That is, solve the following optimization problem over all permutations $\pi$ of $[n]$:
\begin{equation}\label{QAP_1}
{\text{QAP(A,B)}} \quad\quad \quad \quad\min_{\pi } {\sum_{i,j=1}^n{A_{ij}}B_{\pi(i)\pi(j)}},
\end{equation}
where $A_{ij}B_{\pi(i)\pi(j)}$ is the cost inferred by  assigning facility $i$ to location $\pi(i)$ and facility $j$ to location $\pi(j)$. 
QAP has been extensively studied in the past decades, in particular  due to its many real world  applications.
We refer e.g. to  \cite{Burkard98,Burkard09} and references therein for an  exhaustive survey.
QAP is an  NP-hard problem and it cannot even be approximated within a constant factor \cite{Sahni76}. 

However, there exist many special cases which are solvable in polynomial time by exploiting the structure of the data. 
Specifically, we are interested in those ``easy cases" where an optimal solution is known in explicit form and is represented by a
 \textit{fixed} permutation.  These cases have a practical importance in designing heuristics, approximation and enumeration algorithms and they occur when the matrices $A$ and $B$ have an specific ordered structure, like being  Monge, Toeplitz or monotone matrices (see \cite{Woeginger98,Cela98,Demidenko06,Deineko98}, \cite[\S 8.4]{Burkard09} for a survey and the recent works \cite{Cela14,Fogel13}). 
For instance, if $A$ is monotonically nondecreasing (i.e., if both its rows and columns are nondecreasing) and $B$ is monotonically nonincreasing (i.e., if both its rows and columns are nonincreasing), then it is known that the identity permutation is an optimal solution to $\QAP(A,B)$ \cite[Proposition 8.23]{Burkard09}.
Another instance of $\QAP(A,B)$ for which the identity permutation is optimal arises when  $A$  is a Kalmanson matrix, i.e., $A$ is symmetric and satisfies:
$$\max\{ A_{ij}+A_{kl},A_{il}+A_{jk}\}\le A_{ik}+A_{jl}\ \text{ for all } 1\le i<j<k<l\le n,$$ and
$B$ is a symmetric circulant matrix (i.e., $B_{ij}$ depends only on $|i-j|$ modulo $n$) with decreasing generating function (i.e., with $B_{12}\le B_{13}\le \ldots \le B_{1,\lfloor n/2\rfloor +1}$)
 \cite[Theorem 2.3]{Deineko98}. This extends an earlier result of Kalmanson \cite{Kalmanson75}  when $B$ is the adjacency matrix of the  cycle $(1,2,\ldots,n)$,  in which case $\QAP(A,B)$ models the shortest Hamitonian cycle problem 
 on the distance matrix~$A$.

\medskip
The main contribution of this paper is to provide a new class of QAP instances that admit a closed form optimal solution. These instances arise when both matrices $A$ and $B$ have a special ordered structure related to the so-called seriation problem, which is our second problem of interest and is described below.

\medskip
The {\em seriation problem} asks (roughly) to linearly order a set of objects in such a way that similar objects are close to each other and dissimilar objects are far apart. This classical problem was introduced by Robinson \cite{Robinson51} in 1951 for reconstructing the chronology of archeological graves  from information about their similarities. 
Beside archeological dating \cite{Robinson51,Kendall69}  the seriation problem has applications in many other areas including biology \cite{Meidanis98}, machine learning \cite{Ding04}, scheduling \cite{Fulkerson65} and sparse matrix reordering \cite{Barnard93}. 
A more exhaustive list of applications can be found in \cite{Innar10}.

More precisely, an $n\times n$ symmetric matrix $A$ is called a {\em Robinson similarity} matrix if its entries decrease monotonically in the rows and the columns when moving away from the main diagonal, i.e., if 
\begin{equation}\label{eqR}
A_{ik}\le \min\{A_{ij},A_{jk}\} \ \text{  for all } 1\le i\le j\le k\le n.
\end{equation}
Given a symmetric matrix $A=(A_{ij})$,  the seriation problem asks to find a simultaneous reodering of its rows and columns, i.e., a permutation $\pi$ of $[n]$, so that the permuted matrix $A_\pi=(A_{\pi(i)\pi(j)})$ is a Robinson similarity. If such a permutation exists then $A$ is said to be a {\em Robinsonian similarity}.
Analogously a symmetric matrix $B$ is called a {\em Robinson dissimilarity matrix}  if its entries increase monotonically in the rows and columns when moving away from the main diagonal, i.e., if $-B$ is a Robinson similarity; $B$ is then a {\em Robinsonian dissimilarity} if $-B$ is a Robinsonian similarity.
Note that the $0/1$ Robinsonian similarity matrices are exactly the symmetric matrices with the well known \textit{consecutive ones property}. 
 Recall that a matrix $A\in\oR^{n\times n}$ is a Toeplitz matrix if it has constant entries along its diagonals, i.e., if $A_{ij}=A_{i+1,j+1}$ for all $1\le i,j\le n-1$.  
 
Our main result in this paper can then be formulated as follows.  
Assume that $A$ is a Robinson similarity, $B$ is a Robinson dissimilarity and that  one of the two matrices $A$ or $B$ is a Toeplitz matrix. 
Then the identity permutation is an optimal solution for problem $\QAP(A,B)$ (Theorem~\ref{theomain2}). 
From this we  derive  a more general result when both matrices are Robinsonian  (Corollary~\ref{theomain3}). 
Hence our  result uncovers an interesting connection between QAP and the seriation problem and 
  introduces a new class of QAP instances  which is solvable in polynomial time.

Our result can be seen as an analogue for symmetric matrices of the above mentioned result about QAP for monotone matrices, where we replace the monotonicity property by the Robinson property  (which implies unimodal  rows and columns). 

Moreover our result extends two previously known cases. 
The first case is when the Robinson similarity $A$ is a combination of (special) cut matrices and the Robinson dissimilarity $B$ is the Toeplitz matrix $((i-j)^2)$, considered in \cite{Fogel13}. This case (which in fact  has motivated our work) is  discussed  in Section \ref{sec2sum} below  (see Proposition \ref{prep:Rcut-matrices}).
The second class of instances (which was pointed out to us by a referee) is when the similarity matrix $A$ is the  adjacency matrix of the  path $(1,\ldots,n)$ (thus Toeplitz),  and $B$ is a Robinson dissimilarity  which is metric (i.e., $B_{ik}\le B_{ij}+B_{jk} \text{ for all } i,j,k \in [n]$)  and  strongly monotone  (i.e.,   
$B_{ jk}=B_{jl} \Longrightarrow B_{ik}=B_{il}$ and 
$B_{jk}=B_{ik}\Longrightarrow B_{jl}=B_{il}$ for all $1\le i<j<k<l\le n$.)
Then \cite[Lemma 10]{Christopher96} shows that $\QAP(A,B)$ is easily solvable, which also follows from our main result.
Interestingly the above class of  strongly monotone Robinson dissimilarity metrics   plays a central role in the recognition algorithm of \cite{Christopher96} for matrices that are permutation equivalent to Kalmanson matrices.

The rest of the paper is organized as it follows. In Section~\ref{sec2sum} we introduce the seriation problem and discuss some known algorithms to solve it in polynomial time, and we also discuss links to the 2-SUM problem and other instances of QAP. 
Section 3 contains the main result of the paper and in Section 4 we present some applications of it leading to new instances of polynomially solvable QAP's.

\subsubsection*{Notation}
Throughout, $\MP_n$  is the set of permutations of $[n]$, $\MS^n$ is the set of symmetric $n\times n$ matrices and $J$ is the all-ones matrix.
For $A,B\in \oR^{n\times n}$, $\langle A,B\rangle=\Tr(A^\sfT B)=\sum_{i,j=1}^n A_{ij}B_{ij}$ denotes the usual trace inner product on $\oR^{n\times n}$.  
For $\pi\in \MP_n$ and  $A\in \MS^n$, set $A_\pi=(A_{\pi(i)\pi(j)})_{i,j=1}^n\in \MS^n$. 
For $A,B\in\MS^n$ and $\pi,\tau\in\MP_n$, we have:
\begin{equation}\label{eqpi}
{(A_\pi)}_\tau = A_{\pi\tau} \ \text{ and } \langle A,B_\tau\rangle =\langle A_{{\tau}^{-1}},B\rangle.
\end{equation}

Since $A$ is a Robinson dissimilarity if $-A$ is a Robinson similarity, we often refer to $A$ as a Robinson matrix if $A$ or $-A$ is a Robinson similarity matrix.
Clearly, the all-ones matrix $J$ is both a similarity and a dissimilarity Robinson matrix and thus
adding any multiple of $J$ preserves the Robinson property.
Therefore, for a Robinson matrix $A$, we may assume without loss of generality that $A$ is entrywise nonnegative.  
We may also assume without loss of generality that  all diagonal entries are equal in a Robinson similarity
and all diagonal entries are equal to 0 in a Robinson dissimilarity.
On the other hand, observe that when dealing with $\QAP(A,B)$, if one of the two matrices has zero diagonal entries, then the diagonal entries of the other matrix do not play a role.

\section{Seriation and 2-SUM}\label{sec2sum}

Given a set of $n$ objects to order and a similarity matrix $A=(A_{ij})$ which represents the pairwise correlations between the objects, a {\em consistent ordering} of the objects is a permutation $\pi$ of $[n]$ for which the permuted matrix $A_\pi$ is a Robinson similarity, i.e., satisfies the linear constraints (\ref{eqR}). 
The seriation problem consists in finding (it they exist) all possible consistent orderings of  the objects. 

It is straightforward to verify whether a given matrix $A\in\MS^n$ is a Robinson similarity. Moreover one can decide in polynomial time whether $A$ is Robinsonian and if so construct all consistent reorderings. 
Mirkin and Rodin first introduced in 1984 an $O(n^4)$ 
algorithm to recognize Robinsonian similarities. Later, Chepoi and Fichet  \cite{Chepoi97} and Seston  \cite{Seston08} improved the algorithm to respectively $O(n^3)$ and $O(n^2\log n)$ and, very recently, Prea and Fortin  \cite{Prea14} presented an $O(n^2)$ algorithm. 
These algorithms apply to recognize Robinsonian dissimilarities and thus also Robinson similarities as discussed above. They are based on a characterization of Robinsonian matrices  in terms of interval hypergraphs \cite{Fulkerson65}.

A completely different approach to recognize Robinsonian similarities was used by Atkins et al. \cite{Atkins98}, who introduced an interesting spectral sequencing algorithm.
Given a matrix $A\in \MS^n$, its {\em Laplacian matrix}  is the matrix  $L_A=\text{diag}(Ae)-A\in\MS^n$,   where $e$ is the all-ones vector and $\text{diag}(Ae)$ is the diagonal matrix whose diagonal  is given by the vector $Ae$. If $A\ge 0$ then $L_A$ is positive semidefinite and thus its smallest eigenvalue is $\lambda_1(L_A)=0$. Moreover, the second smallest eigenvalue $\lambda_2(L_A)$ is positive if $A$ is irreducible.
 The \textit{Fiedler vector} $y_F \in \oR^n$ of $A$ is the eigenvector corresponding to the second smallest eigenvalue $\lambda_2(L_A)$ of $L_A$, which is also known as the \textit{Fiedler value}.
The spectral algorithm of Atkins et {al}. \cite{Atkins98} relies  on the following properties of the Fiedler vector  of a Robinson similarity matrix.

\begin{theorem}\cite{Atkins98} \label{prep:Atkins R-matrices}
If $A \in \mathcal{S}^n$ is a Robinson similarity then it has a monotone Fiedler vector $y_F$, i.e., satisfying: 
$y_F(1)\le \ldots \le y_F(n)$ or $y_F(n)\le \ldots \le y_F(1)$.
\end{theorem}

\begin{theorem}\cite{Atkins98}\label{prep:Atkins Fiedler vector}
Assume that  $A \in \mathcal{S}^n$ is a Robinsonian similarity, that  its Fiedler value is simple and  that the  Fiedler vector $y_F$ has no repeated entries. Let $\pi$ be  the permutation induced by sorting monotonically the values of $y_F$ (in increasing or decreasing order). Then the matrix $A_\pi$ is a Robinson similarity matrix.
\end{theorem}

In other words, the above results show  that sorting monotonically the Fiedler vector $y_F$ reorders $A$ as a Robinson similarity. The complexity of the algorithm of \cite{Atkins98} in the general case is given by $O(n(T(n)+n\log n))$, where $T(n)$ is the complexity of computing (approximately) eigenvalues of an $n\times n$ symmetric matrix. 

Barnard et {al.}  \cite{Barnard93} had earlier used the same spectral algorithm to solve the {2-SUM problem}.
Given a matrix $A\in\MS^n$, the 2-SUM problem is the special instance $\QAP(A,B)$ of QAP, where the distance matrix is $B=((i-j)^2)_{i,j=1}^n$. That is,
\begin{equation}\label{2-SUM}
\text{2-SUM}\quad\quad\quad\quad \min_{\pi \in \mathcal{P}_n} {\sum_{i,j=1}^n{A_{\pi(i)\pi(j)}(i-j)^2}}.
\end{equation}
The  authors of \cite{George97} show that 2-SUM   is an NP-complete problem and they use the above spectral method of reordering the Fiedler vector of $A$ to produce a heuristic solution for problem (\ref{2-SUM}).
This in turn permits to bound  important matrix structural parameters, like  envelope-size and bandwidth.
However, no assumption is made on the structure of the matrix $A$. 
As observed in~\cite{George97,Fogel13}, the following fact can be used to motivate the spectral approach for  2-SUM. If we define the vectors  $x=(1,\ldots,n)^\sfT, x_\pi=(\pi(1),\ldots,\pi(n))^\sfT \in \oR^n$ for $\pi\in\MP_n$, then  $\sum_{i,j=1}^n A_{ij}(\pi(i)-\pi(j))^2 = (x_\pi)^\sfT  L_Ax_\pi$ and thus computing the Fiedler value arises as a natural continuous relaxation for the 2-SUM problem (\ref{2-SUM}). 

Fogel et {al.} \cite{Fogel13}  point out an interesting  connection between the  seriation  and 2-SUM problems. 
They consider  a special class of Robinson similarity matrices for which they  can show that the identity permutation is optimal for the 2-SUM problem (\ref{2-SUM}). Namely they consider the following {\em cut matrices}:
given two integers $1\le u\le v\le n$, the {\em cut matrix} $\CUT(u,v)$ is the symmetric $n\times n$ matrix  with $(i,j)$-th entry 1 if $u\le i,j\le v$ and 0 otherwise. Clearly each cut matrix is a Robinson similarity and thus  conic combinations of cut matrices as well. The following result is shown in \cite{Fogel13}, for which we give a short proof.

\vspace{5ex}

\begin{proposition}\cite{Fogel13}\label{propFogel}
\label{prep:Rcut-matrices}
If $A\in \MS^n$ can be written as a conic combination of  cut matrices, then the identity permutation  is optimal for the 2-SUM problem (\ref{2-SUM}).
More generally if, for some $\pi\in\MP_n$,  $A_\pi$ can be written as a conic combination of  cut matrices, then $\pi$   is optimal for the 2-SUM problem (\ref{2-SUM}).
\end{proposition}

\begin{proof}
First we show the result when $A$ is a cut matrix. Say, $A=\CUT(u,v)$ and set $t=v-u+1$. Then, we need to show that
{\bf $\sum_{i,j=u}^v (\pi(i)-\pi(j))^2\ge \sum_{i,j=u}^v (i-j)^2$}  for any permutation $\pi\in\MP_n$.
Suppose that $\pi$ maps the elements of the interval $[u,v]$ to the elements of the set $\{i_1,\ldots,i_t\}$, 
ordered as 
$1\le i_1< \ldots <i_t\le n$.
Because the left-hand side of the above inequality involves all pairs of indices in the interval $[u,v]$, equivalently, we need to show that $$\sum_{1\le r<s\le t} (i_r-i_s)^2 \ge \sum_{1\le r<s\le t}(r-s)^2.$$
Now the latter easily follows from the fact that $|i_r-i_s|\ge |r-s|$ for all $r,s$. Hence we have shown that the identity permutation is an optimal solution of problem (\ref{2-SUM}) when $A$ is a cut matrix and this easily implies that this also holds when $A$ is a conic combination of cut matrices. The second statement follows as a direct consequence.
 \end{proof}

In other words, the above result shows that for similarity matrices as in Proposition \ref{prep:Rcut-matrices}, any permutation reordering $A$ as a Robinson matrix also solves (\ref{2-SUM}).
As not every Robinson similarity is a conic combination of cut matrices, this raises the question whether the above result extends to the case when $A$ is an arbitrary Robinson similarity matrix.

There is a second possible way in which one may want to generalize the result of Proposition \ref{propFogel}.
Indeed,  the 2-SUM problem (\ref{2-SUM}) is the instance of $\QAP(A,B)$, where the distance matrix is $B=((i-j)^2)_{i,j=1}^{n}$, which turns  out to be  a Toeplitz dissimilarity Robinson matrix. 
In fact there are many other interesting classes of QAP whose distance matrix $B$ is a Toeplitz Robinson dissimilarity matrix. For instance, $\QAP(A,B)$ models the minimum linear arrangement (aka 1-SUM) problem when
$B=(|i-j|)_{i,j=1}^n$ and, more generally, the $p$-SUM problem when we have $B=(|i-j|^p)_{i,j=1}^{n}$ (for  $p\ge 1$),  and $\QAP(A,B)$ models the minimum bandwith problem when $A$ is the adjacency matrix of a graph and $B$ is of the form  $B^\Delta_n$,  as defined in relation (\ref{eqBDelta}) below.
For more details on these and other graph (matrix) layout problems with pratical impact we refer to the survey \cite{Diaz02} and references therein.

This thus raises  the further question whether the result of Proposition \ref{propFogel}  extends to instances of $\QAP(A,B)$, where $B$ is an arbitrary Toeplitz dissimilarity matrix. This is precisely what we do in this paper.
We remove both assumptions on $A$ and $B$ and show that the identity permutation is optimal for $\QAP(A,B)$ when $A$ is any 
Robinson similarity and $B$ is any Robinson dissimilarity, assuming that $B$ (or $A$)    has a Toeplitz structure.

\section{The main result}
Let $A$ be a Robinson similarity matrix and let $B$ be a Toeplitz Robinson dissimilarity matrix. 
The first  key (easy) observation  is that we can decompose $B$ as a conic combination of 0/1 Toeplitz Robinson dissimilarities.
Given an integer $\Delta \in [n]$, we define the symmetric matrix 
$B^\Delta_n\in \MS^n$ with entries
\begin{equation}\label{eqBDelta}
\left( B_n^{\Delta}\right)_{ij}=
\begin{cases} 
1 \quad \text{if } |i-j|\geq n-\Delta \\
0 \quad \text{else }
\end{cases}  \qquad \text{ for } \  i,j=1,\dots,n.
\end{equation}
Note that for $\Delta = n$, we have that $B^{\Delta}_n = J$.
Clearly,  each matrix $B^\Delta_n$ is a Toeplitz matrix and a Robinson dissimilarity. In fact all Toeplitz Robinson dissimilarities can be decomposed in terms of these matrices $B^\Delta_n$.

\begin{lemma}\label{lemToeplitz}
Let $B\in \MS^n$ be a Toeplitz matrix and let $\beta_0,\ldots,\beta_{n-1}\in \oR$ such that $B(i,j)=\beta_k$ for all $i,j\in [n]$ with $|i-j|=k$ for $0\le k\le n$. 
Then,
\begin{equation}
B=\beta_0J+\sum_{k=1}^{n-1}(\beta_k-\beta_{k+1})B^{n-k}_n.
\end{equation}
If moreover $B$ is a Robinson dissimilarity, i.e., if $\beta_0=0\le \beta_1\le \ldots\le \beta_{n-1}$, then $B$ is a conic combination of the matrices $B^\Delta_n$ (for $\Delta=1,\ldots,n-1$).
\end{lemma}

\begin{proof}
Direct verification.
\end{proof}

Our main result, which we show in this section, is that the identity permutation is optimal for $\QAP(A,B_n^{\Delta})$ for any integer $1 \leq \Delta \leq n-1$.
We will mention its application to $\QAP(A,B)$ in Section \ref{secapplication}.

\begin{theorem}\label{theomain1}
Let $A\in \MS^n$ be a Robinson similarity matrix and let 
$\Delta \in [n-1]$. Then, for any permutation $\pi$  of $[n]$, we have:
\begin{equation}\label{eq1}
\langle A_\pi,B^\Delta_n\rangle
\geq \langle A, B^\Delta_n\rangle.
\end{equation}
\end{theorem}

\begin{proof}
Let $E^\Delta_n$ denote the support of the matrix $B^\Delta_n$, i.e., the set of upper triangular positions where $B^\Delta_n$ has a nonzero entry. That is:
\begin{equation*}
E^\Delta_n=\{\{i,n-\Delta+j\}: 1\le i\le j\le \Delta\}.
\end{equation*}
Then we can reformulate the inner products in (\ref{eq1}) as 
\begin{equation*}
\langle A_\pi,B^{\Delta}_{n}\rangle = \sum_{i,j=1}^{n}{A_{\pi(i),\pi(j)}\left( B^{\Delta}_{n}\right) _{i,j}} = 2 \sum_{\{i,j\}\in E^\Delta_n} A_{\pi(i),\pi(j)},
\end{equation*}

\begin{equation*}
\langle A, B^{\Delta}_{n} \rangle = \sum_{i,j=1}^{n}{A_{i,j}\left( B^{\Delta}_{n}\right) _{i,j}} = 2 \sum_{\{i,j\}\in E^\Delta_n} A_{i,j},
\end{equation*}
and (\ref{eq1}) is equivalent to the following inequality:
\begin{equation}\label{eq2}
\sum_{\{i,j\}\in E^\Delta_n} A_{\pi(i),\pi(j)} \ge \sum_{\{i,j\}\in E^\Delta_n} A_{ij}.
\end{equation}
We show the inequality (\ref{eq2}) using  induction on $n\ge 2$.
The base case $n=2$ is trivial, since then $\Delta=1$ and both summations in (\ref{eq2}) are  identical.
 We now assume that the result holds for $n-1$ and we show that it also holds for $n$. 
 For the remaining of the proof, we fix a Robinson similarity matrix $A\in \MS^n$, an integer $\Delta \in [n-1]$ and a permutation $\pi$ of $[n]$. Moreover we let $k\in [n]$ denote the index such that  $n=\pi(k)$. 
 
The key idea in the proof is to show that  there exist a subset $F\subseteq E^\Delta_n$  and a permutation $\tau$ of $[n]$ satisfying  the following properties:
\begin{itemize}
\item[(C1)] $|F|=\Delta$,
\item[(C2)] the indices $\min\{\pi(i),\pi(j)\}$ for the pairs $\{i,j\}\in F$ are pairwise distinct,
\item[(C3)] $\tau(n) = k$ and the set $R:=E^\Delta_n\setminus F$ satisfies 
\begin{equation*}
R=\tau(E^{\Delta -1}_{n-1}):=\{\{\tau(i),\tau(j)\}: \{i,j\}\in E^{\Delta-1}_{n-1}\}.
\end{equation*}
\end{itemize}
Here the set  $E_{n-1}^{\Delta-1}$ is the support of the matrix $B_{n-1}^{\Delta-1}$, defined by
$$E_{n-1}^{\Delta-1}=\{\{i,n-\Delta+j\}: 1\le i\le j\le \Delta-1\},$$
so that $E_n^{\Delta}$ is partitioned into the two sets  $\{\{i,n\}: 1\le i\le \Delta\}$ and $E^{\Delta-1}_{n-1}$. 

In a first step, we show (in Claim \ref{claim1} below) that if we can find  a set $F$ and a permutation $\tau$ satisfying (C1)-(C3), then we can conclude the proof of the inequality (\ref{eq2}) using the induction assumption.
The proof relies on the following idea: we split the summation
\begin{equation}\label{eqA3}
\Sigma_\pi(A) := \sum_{\{i,j\}\in E^\Delta_n} A_{\pi(i),\pi(j)}
\end{equation} 
into two terms, obtained by summing over the set $F$ and over its complement $R$, and we show that the first term is at least $\sum_{i=1}^\Delta A_{in}$ (using the conditions (C1)-(C3)) and that the second term is at least $\sum_{\{i,j\}\in E^{\Delta-1}_{n-1}} A_{ij}$ (using the induction assumption applied to the smaller Robinson similarity $(A_{ij})_{i,j=1}^{n-1}$).

\begin{claim}\label{claim1}
Assume that there exist a set $F\subseteq E^\Delta_n$ and a permutation $\tau$ of $[n]$ satisfying (C1)-(C3), then the inequality (\ref{eq2}) holds.
\end{claim}

\begin{proof} {\em (of Claim \ref{claim1})}.
Let us decompose the summation $\Sigma_\pi(A)$ from (\ref{eqA3}) 
as the sum of the two terms:
\begin{equation}\label{eq0}
\Sigma_\pi(A)=\Sigma_{\pi,F}(A)+\Sigma_{\pi,R}(A),
\end{equation} 
where we set:
\begin{equation*}
\Sigma_{\pi,F}(A):= \sum_{\{i,j\}\in F} A_{\pi(i),\pi(j)}, \quad
\Sigma_{\pi,R}(A):=  \sum_{\{i,j\}\in R} A_{\pi(i),\pi(j)}.
\end{equation*}
We now bound each term separately. 
First we consider the term $\Sigma_{\pi,F}(A)$.
As  $A$ is a Robinson similarity matrix, it follows that for all indices $i,j\in [n]$:
\begin{equation*}
A_{ij}\ge A_{n, \min\{i,j\}}.
\end{equation*}
Hence, we can deduce:
\begin{equation}\label{eq1a}
\Sigma_{\pi,F}(A)= \sum_{\{i,j\}\in F} A_{\pi(i),\pi(j)} \geq \sum_{\{i,j\}\in F} A_{n, \min \{\pi(i),\pi(j)\}}
\ge \sum_{i=1}^\Delta A_{n,i},
\end{equation}
where for the right most inequality we have used the conditions (C1),(C2) combined with the fact  that $A$ is a Robinson similarity.

We now consider the second term $\Sigma_{\pi,R}(A)$. 
Define the permutation $\sigma= \pi\tau$. Then, by (C3), we have that 
$\sigma(n)=\pi(k)=n$ and thus $\sigma (E_{n-1}^{\Delta-1})= \pi(\tau(E_{n-1}^{\Delta-1}))=\pi(R)$. 
We can then write:
\begin{equation}\label{eq2a}
\Sigma_{\pi,R}(A) = \sum_{\{i,j\}\in E^{\Delta-1}_{n-1}} A_{\sigma(i),\sigma(j)}.
\end{equation}
As $\sigma(n)=n$, the permutation $\sigma$ of $[n]$ induces a permutation $\sigma{'}$ of $[n-1]$. 
We let  $A{'}, B' \in \MS^{n-1}$ denote the principal submatrices obtained by deleting the row and column indexed by $n$ in $A$ and in $B^\Delta_n$, respectively. Then $A'$ is again a Robinson similarity matrix (now of size $n-1$) and $B'=B^{\Delta-1} _{n-1}$ is supported by the set $E_{n-1}^{\Delta-1}$. 
Then, using the induction assumption applied to $A'$, $\Delta-1$ and $\sigma'$, we obtain:
\begin{equation}
\Sigma_{\sigma{'}}(A{'}):=
\sum_{\{i,j\} \in E_{n-1}^{\Delta-1}}{A'_{\sigma(i),\sigma(j)}} \geq \Sigma_{\text{id}}(A^{'}) := \sum_{\{i,j\} \in E_{n-1}^{\Delta-1}}{A'_{i,j}}.
\end{equation}
Using (\ref{eq2a}), we get:
\begin{equation}\label{eq2b}
\Sigma_{\pi,R}(A) = \Sigma_{\sigma{'}}(A{'})  \geq \Sigma_{\text{id}}(A^{'})= \sum_{\{i,j\} \in E_{n-1}^{\Delta-1}}{A_{i,j}}.
\end{equation}
Finally, combining (\ref{eq0}),(\ref{eq1a}) and (\ref{eq2b}), we get the desired inequality:
\begin{equation*}
\Sigma_{\pi}(A) \geq \sum_{i=1}^\Delta A_{i,n} + \sum_{\{i,j\}\in E^{\Delta-1}_{n-1}} A_{ij} =\sum_{\{i,j\}\in E^\Delta_n} A_{ij},
\end{equation*}
which concludes the proof of Claim  \ref{claim1}. \end{proof}

In a second step, we formulate (in Claim \ref{claim2} below) two new conditions (C4) and (C5) which together with (C1),(C2) imply (C3). These two conditions will be   simpler to check than (C3).

\begin{claim}
\label{claim2}
Assume that the sets  $F\subseteq E^\Delta_n$ and $R :=E^\Delta_n\setminus F$ satisfy  the conditions (C1),(C2) and  the following two conditions:
\begin{itemize}
\item[(C4)]
no pair in the set $R$ contains the element $k$,
\item[(C5)]
no  pair $\{i,n-\Delta+i\}$ with $1\le i\le \Delta$ and 
$k+1-n+\Delta\le i\le k-1$ belongs to the set $R$.
\end{itemize}
Define the  permutation $\tau=(k,k+1,\ldots,n)$.
Then, we have $R=\tau(E^{\Delta-1}_{n-1})$.
\end{claim}

\begin{proof} {\textit{(of Claim \ref{claim2}).}}
By (C1), $|F|=\Delta$, thus  $R$ has the same cardinality as the set $\tau(E^{\Delta-1}_{n-1})$ and therefore it suffices to show the inclusion
$R\subseteq \tau(E^{\Delta-1}_{n-1})$.
For this consider a pair $\{i,n-\Delta+j\}$ in $R$, where $1\le i\le j\le \Delta$.
We show that $i=\tau(r)$ and $j=\tau(s)$ for some $1\le r\le s\le \Delta-1$.
In view of (C4), we shall define $r,s$ depending whether $i$ lies before or after $k$, getting the following cases:
\begin{itemize}
\item[1)] $i\le k-1$, then $r=i$ and $i=\tau(i)$,
\item[2)] $i\ge k+1$, then $r=i-1$ and $i=\tau (i-1)$.
\end{itemize}
We do the same for index $j$, getting the following cases:
\begin{itemize}
\item[a)] $n-\Delta+j \le k-1$, then $s=j$ and $n-\Delta+j=\tau(n-\Delta+j)$,
\item[b)] $n-\Delta+j\ge k+1$, then $s=j-1$ and $n-\Delta+j=\tau(n-\Delta+j-1)$.
\end{itemize}
It remains only to check that $1\le r\le s\le \Delta-1$ holds. For this, we now discuss all possible combinations for indices $i$ and $j$ according to the above cases:
\begin{itemize}
\item[1a)] Then, $r=i$ and $s=j$. Since $1 \leq i \leq j \leq \Delta$, we only have to check that $j \leq \Delta-1$. Indeed, if $j=\Delta$, then from a) we get $n \leq k-1$, which is impossible.
\item[1b)] Then, $r=i$ and $s=j-1$. It suffices to check that $r\le s$, i.e., $i\ne j$. Indeed, assume that $i=j$. Then, the pair $\{i,n-\Delta+i\}$ belongs to $R$ with $i\le k-1$ (as we are in case 1) for index $i$) and $i\ge k+1-n+\Delta$ (as we are in case b) for index $j$), which contradicts the condition (C5).
\item[2a)] Then, $r=i-1$ and $s=j$. It suffices to check that $r \geq 1$ and $s \leq \Delta-1$. The first one holds since $i\geq 2$ as we are in case 2) for index $i$. The second one is also true, since we are in case a) for index $j$ and $k \leq n$.
\item[2b)] Then, $r=i-1$ and $s=j-1$. It suffices to check that $r \geq 1$, which holds since we are in case 2) for index $i$ and thus $i \geq 2$.
\end{itemize}
Thus we have shown that $R\subseteq \tau(E^{\Delta-1}_{n-1})$, which concludes the proof of Claim \ref{claim2}.
\end{proof}

In view of Claims \ref{claim1} and \ref{claim2}, in order to conclude the proof of the inequality (\ref{eq2}) it suffices to find a set $F\subseteq E^{\Delta}_n$ and a permutation $\tau$ of $[n]$ satisfying the conditions (C1), (C2), (C4) and (C5). 

For  the permutation $\tau$,  we choose  $\tau= (k,k+1,\ldots,n)$ as in Claim \ref{claim2}, thus $\tau(n)=k$.
It remains to construct the set $F$.  This is the last step in the proof which is a bit technical.

\medskip
The following  terminology will be useful, regarding the pairs $\{i,n-\Delta+j\}$ (for $1\le i\le j\le \Delta$) in the set $E^\Delta_n$. We refer to the pairs $\{i,n-\Delta+i\}$ (with $ 1\leq i\leq \Delta$) as the \textit{diagonal pairs}. 
Furthermore, for each $1 \leq i_0 \leq \Delta$, we refer to the pairs  $\{i_0,n-\Delta+j\}$ (with $ i_0+1 \leq j\leq\Delta$) as the \textit{horizontal pairs} on row $i_0$, meaning the pairs of $E_{n}^{\Delta}$ in the row indexed by $i_0$. 
Finally, for each $k_0=n-\Delta+j_0$ such that $1 \leq j_0 \leq \Delta$, we refer to the pairs $\{i,n-\Delta+j_0\}$ (with $1 \leq i \leq j_0-1$) as the \textit{vertical pairs} on column $k_0$, meaning the pairs of $E_{n}^{\Delta}$ in the column indexed by $k_0$.
Note that, in both horizontal and vertical pairs, the diagonal pair $\{i,n-\Delta+i\}$ is not included.
As an illustration see Figure~\ref{fig:pairs}.
\begin{figure}[!htbp]
	\centering	
	\includegraphics[scale = 0.5]{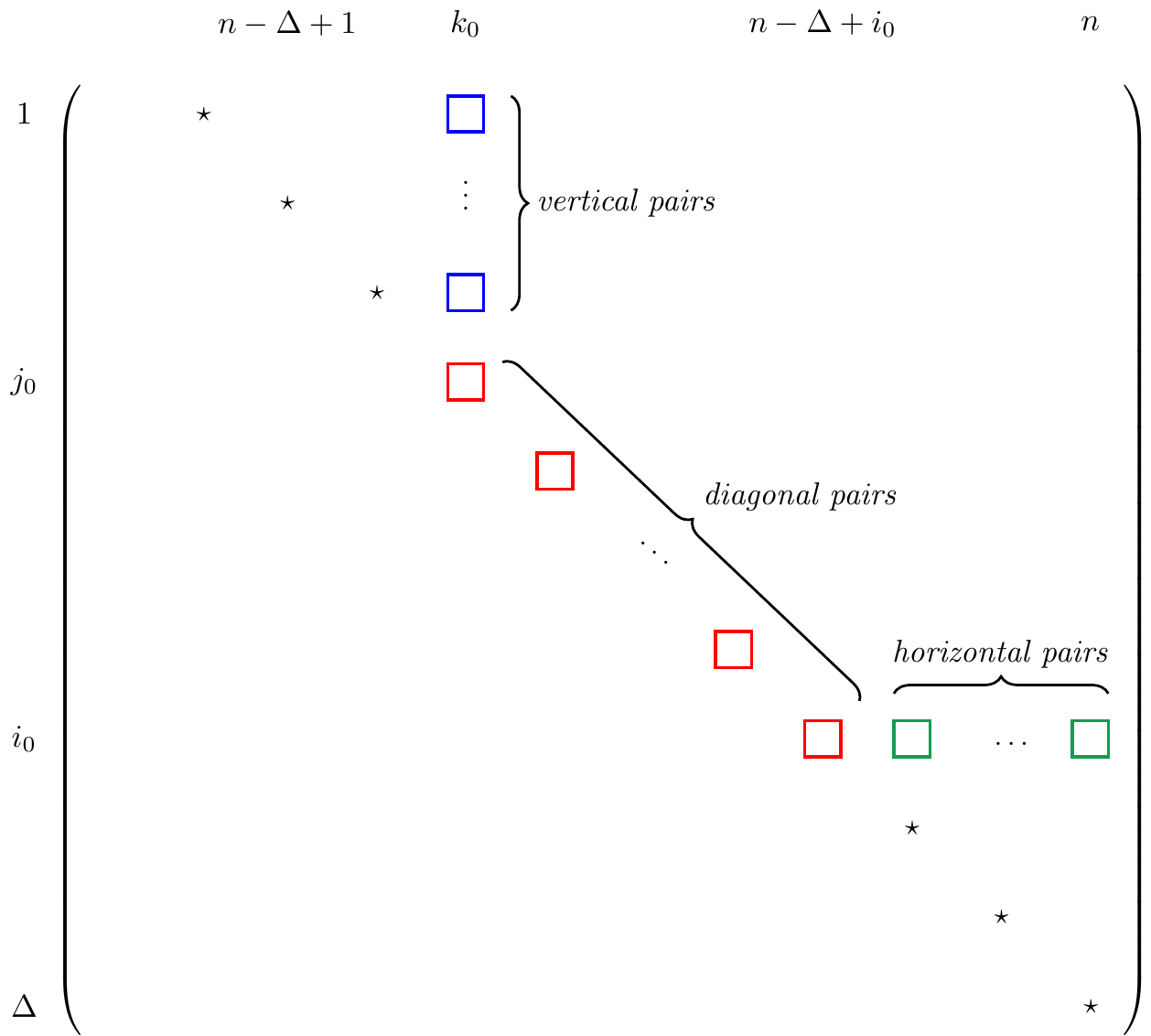}
	\caption{Vertical, diagonal and horizontal pairs in the set $E_{n}^{\Delta}$}
	\label{fig:pairs}
\end{figure}

Moreover, we denote by:
\begin{equation*}
U=U_R\cup U_C, \text{ where } U_R=\{1,\ldots,\Delta\},\quad U_C=\{n-\Delta+1,\ldots,n\},
\end{equation*}
the set consisting of the row and column indices for the nonzero entries of the matrix $B^\Delta_n$. 

In the rest of the proof we indicate how to construct the set $F$.
In view of (C4), the set $F$ must contain all the pairs in $E^\Delta_n$ that contain the index $k$.
Moreover, in view of (C5),  $F$ must contain all the diagonal pairs, except those coming before the position $(k+1-n+\Delta,k)$ on  column $k$ (if it exists) and those coming after the diagonal position $(k, n-\Delta+k)$ on row $k$ (if it exists).
Hence  we must discuss 
 depending whether the index $k$ belongs to the sets $U_R$ and/or $U_C$.
Namely we consider the following four cases: (1) $k\not\in U_R\cup U_C$, (2) $k\in U_R\setminus U_C$, (3) $k\in U_C\setminus U_R$, and (4) $k\in U_R\cap U_C$.
In each of these cases, we  define  the set $F$ which, by construction, will satisfy
the conditions (C1), (C4) and (C5). 
Hence it will remain only to verify that condition (C2) holds in each of the four cases and this is what we do below.

\medskip \noindent
{\bf Case (1):} $k\not\in U_R\cup U_C$.\\
Then, $\Delta +1\le k\le n-\Delta$, which implies $\Delta \le (n-1)/2$. 
In this case we define $F$ as the set of all diagonal pairs (see Figure \ref{fig:case1} at page \pageref{fig:case1}), namely:
\begin{equation*}
F=\{\{i,n-\Delta+i\}: 1\le i\leq \Delta\}
\end{equation*}

\begin{figure}[!h]
\centering
\begin{subfigure}{.5\textwidth}
  \centering
  \includegraphics[width=.8\linewidth]{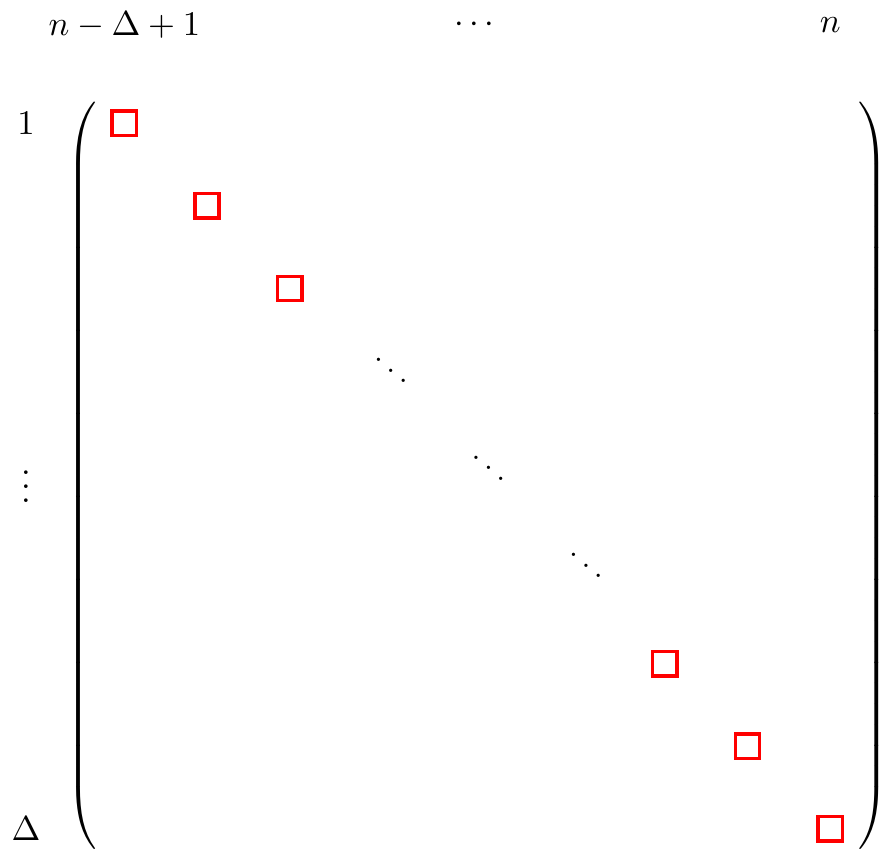}
  \caption{Definition of set $F$ for case (1)}
  \label{fig:case1}
\end{subfigure}%
\begin{subfigure}{.5\textwidth}
  \centering
  \includegraphics[width=.8\linewidth]{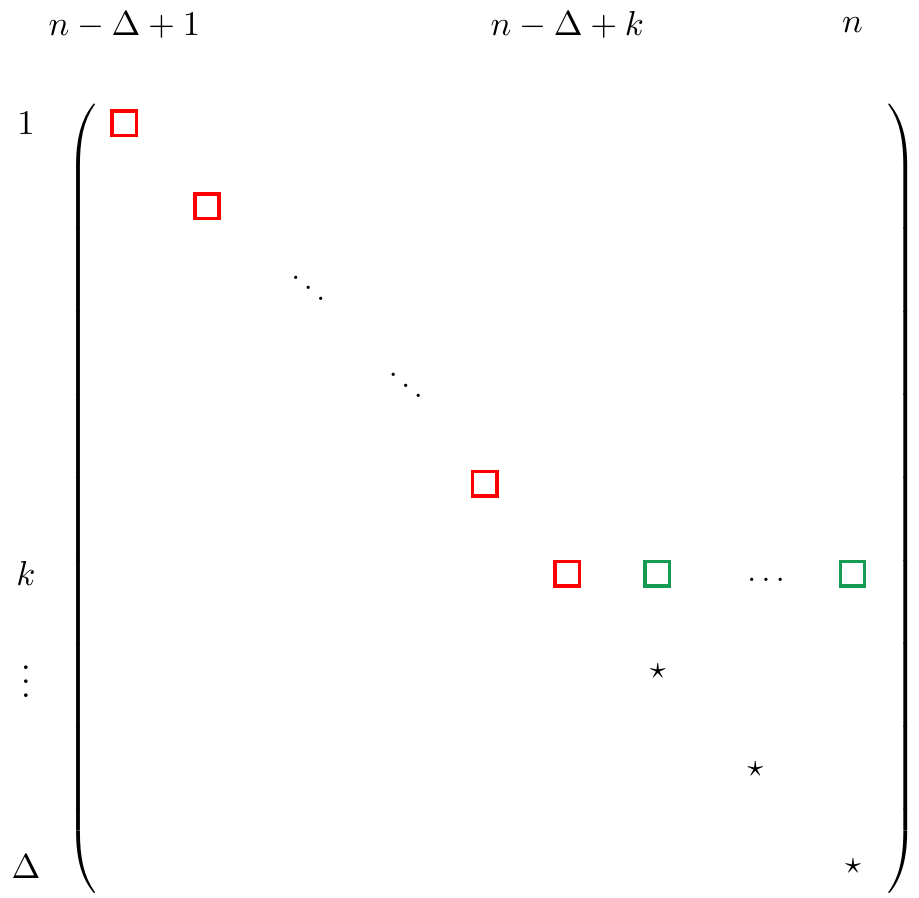}
  \caption{Definition of set $F$ for case (2)}
  \label{fig:case2}
\end{subfigure}
\bigskip
\begin{subfigure}{.5\textwidth}
  \centering
  \includegraphics[width=.8\linewidth]{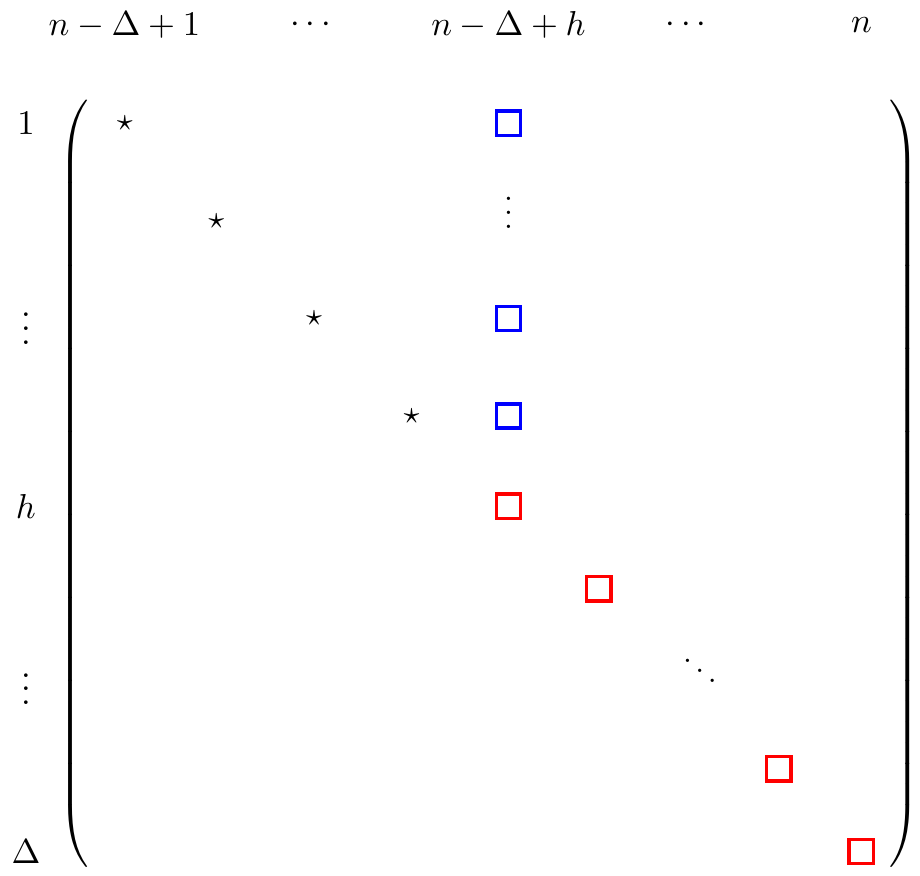}
  \caption{Definition of set $F$ for case (3)}
  \label{fig:case3}
\end{subfigure}%
\begin{subfigure}{.5\textwidth}
  \centering
  \includegraphics[width=.8\linewidth]{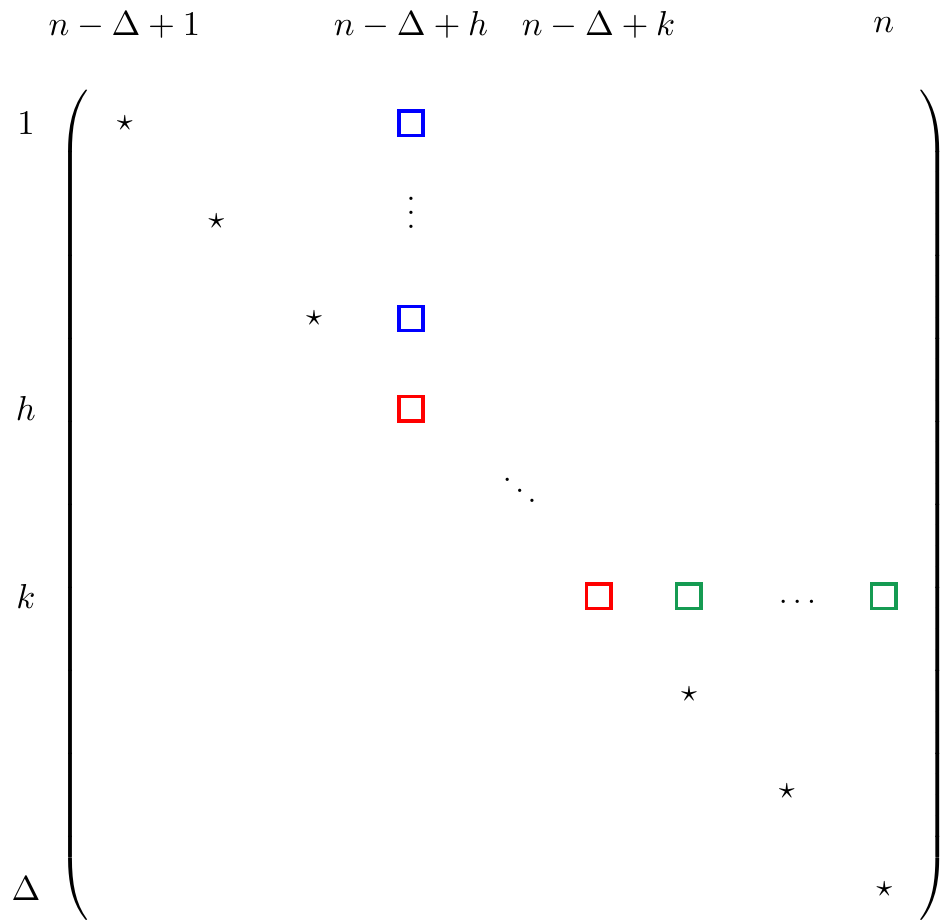}
  \caption{Definition of set $F$ for case (4)}
  \label{fig:case4}
\end{subfigure}
\caption{Different definition of set $F$ for all 4 cases}
\label{fig:Fset}
\end{figure}

To see  that (C2) holds, let $r\ne s\in [\Delta]$ and we have to show that 
$\min\{\pi(r),\pi(n-\Delta+r)\} \neq \min\{\pi(s),\pi(n-\Delta+s)\}.$

This is clear since the four indices $\pi(r),$ $ \pi(n-\Delta+r)$, $\pi(s),$ and $\pi(n-\Delta+s)$ are pairwise distinct.
Indeed, if equality ${\pi(r)=\pi(n-\Delta+s)}$ would hold, this would  imply the inequalities: ${n-\Delta+1\le r=n-\Delta +s\le \Delta}$ and thus ${\Delta \ge (n+1)/2}$, a contradiction.

\bigskip\noindent
{\bf Case (2):}  $k\in U_R\setminus U_C$.\\
Then, $1\le k\le \Delta$ and $k\le n-\Delta$. 
In this case we let $F$ consist of the diagonal pairs till position $(k,n-\Delta+k)$ and then of the horizontal pairs on the $k$-th row (see Figure \ref{fig:case2} at page \pageref{fig:case2}), namely:
\begin{equation*}
F=\{\{i,n-\Delta+i\}: 1\le i\le k\} \cup \{\{k,n-\Delta +i\}: k+1\le i\le \Delta\}.
\end{equation*}
In order to check that (C2) holds, we consider  the following three cases:
\begin{itemize}
\item For $r\ne s \in [k]$, 
we get ${\min\{\pi(r),\pi(n-\Delta+r)\}} \ne \min\{\pi(s),\pi(n-\Delta+s)\}$, since the four indices 
$\pi(r), {\pi(n-\Delta +r)}, \pi(s), \pi(n-\Delta +s)$ are pairwise distinct (using the fact that  $k \le n-\Delta$).
\item For $r\ne s\in \{k+1,\ldots,\Delta\}$, 
using the fact that  $\pi(k)=n$, $\min \{\pi(k),\pi(n-\Delta +r)\}=\pi (n-\Delta +r) \ne {\min\{\pi(k),\pi(n-\Delta+s)\}=\pi(n-\Delta+s)}$.
\item Finally, for $r\in [k]$ and $s\in\{k+1,\ldots, \Delta\}$, we have that $\min\{\pi(r),\pi(n-\Delta+r)\}\ne \min\{\pi(k),\pi(n-\Delta+s)\}= \pi(n-\Delta+s) $. Indeed,
$\pi(r)\ne \pi(n-\Delta +s)$ (since otherwise this would imply that
$n-\Delta+k+1\le r=n-\Delta +s\le k$ and thus $n+1\le \Delta$, a contradiction) and it is clear that 
$\pi(n-\Delta+r)\ne \pi(n-\Delta +s)$. 
The case for $s\in [k]$ and $r\in\{k+1,\ldots, \Delta\}$ is symmetric.
\end{itemize}

\medskip\noindent
{\bf Case (3):} $k\in U_C\setminus U_R$.\\
This case corresponds to $k=n-\Delta+h$, where $1 \leq h \leq \Delta$. 
Then, since $k$ belongs to the column indices, we have that $n-\Delta +h=k\ge \Delta +1$, which implies 
$\Delta\le (n+h-1)/2$. 

In this case we let $F$ consists of the vertical pairs on the $k$-th column and of the diagonal pairs from position $(k,n-\Delta+k)$ (see Figure \ref{fig:case3} at page \pageref{fig:case3}), namely:
\begin{equation*}
F=\{\{i,k\}: 1\leq i \leq h-1\} \cup \{\{i,n-\Delta+i\}: h \leq i \leq \Delta\}.
\end{equation*}

To see that  (C2) holds, we consider the following three cases.
\begin{itemize}
\item For $r\ne s \in [h-1]$, $\pi(r)= \min\{\pi(r),\pi(k)\} \ne \min \{\pi(s),\pi(k)\}=\pi(s)$.
\item For $r\ne s \in \{h,\ldots,\Delta\}$, 
${\min\{\pi(r),\pi(n-\Delta+r)\}} \ne \min\{\pi(s),\pi(n-\Delta+s)\}$, since  the four indices 
${\pi(r)},{\pi(n-\Delta+r)}, {\pi(s)}, \pi(n-\Delta +s)$ are pairwise distinct.
Indeed, $\pi(r)=\pi(n-\Delta +s)$ would imply: 
$n-\Delta+h\le r=n-\Delta+s\le \Delta$ and thus $\Delta \ge (n+h)/2$, a contradiction.
\item For $r\in [h-1]$ and $s\in \{h,\ldots,\Delta\}$, $\min\{\pi(r),\pi(k)\}=\pi(r) \ne \min\{\pi(s),\pi(n-\Delta+s)\}$, since the indices $\pi(r),\pi(s),\pi(n-\Delta+s)$ are pairwise distinct.
Indeed, equality $\pi(r)=\pi(n-\Delta+s)$ would imply that
${n-\Delta +h\le r=n-\Delta+s\le h-1}$ and thus $\Delta\ge n+1$, a contradiction.
\end{itemize}

\medskip\noindent
{\bf Case (4):} $k\in U_R\cap U_C=\{n-\Delta+1,\ldots,\Delta\}$.\\
This case corresponds to $k=n-\Delta+h$, where $1 \leq h \leq 2\Delta-n$. Moreover, we have $\Delta \ge (n+1)/2$.

Then we let $F$ consist of the vertical pairs on the $k$-th column, of the diagonal pairs from position $(h,n-\Delta+h)$ to position $(k,n-\Delta+k)$, and of the horizontal pairs on the $k$-th row (see Figure \ref{fig:case4} at page \pageref{fig:case4}), namely:
\begin{align*}
F=\{\{i,k\}: 1\leq i\leq h\} \  \cup \
\{\{i,n-\Delta +i\}: h+1\leq i\leq k\}\   \\
\cup \  \{\{k,n-\Delta+i\}: k+1\leq i\leq \Delta\}.
\end{align*}

In order  to check condition (C2), as $F$ consists of the union of three subsets we need to consider the following six cases.
\begin{itemize}
\item 
For $r\ne s\in [h]$, 
$\min\{\pi(r),\pi(k)\}=\pi(r)\ne \min\{\pi(s),\pi(k)\}=\pi(s)$.
\item
For $r\ne s\in \{h+1,\ldots,k\}$, ${\min\{\pi(r),\pi(n-\Delta+r)\}} \ne {\min\{\pi(s),\pi(n-\Delta+s)\}}$, since $\pi(r), \pi(n-\Delta+r), \pi(s)$ and $\pi(n-\Delta+s)$ are pairwise distinct. Indeed, equality $\pi(r)=\pi(n-\Delta+s)$ would imply $r=n-\Delta+s$ and thus 
$k+1=n-\Delta +h+1\le n-\Delta+s=r\le k$, yielding a contradiction.
\item
For $r\ne s \in\{k+1,\ldots,\Delta\}$, it holds
$\min\{\pi(k),\pi(n-\Delta+r)\}=\pi(n-\Delta +r)\ne 
\min\{\pi(k),\pi(n-\Delta+s)\}=\pi(n-\Delta +s).$
\item
For $r \in [h]$ and $s\in \{h+1,\ldots,k\}$, $\min\{\pi(r),\pi(k)\}=\pi(r) \ne \min\{\pi(s),\pi(n-\Delta+s)\}$, since the indices $\pi(r),\pi(s),\pi(n-\Delta+s)$ are pairwise distinct.
Indeed,
$\pi(r)=\pi(n-\Delta+s)$ would imply that $n-\Delta+h+1\le n-\Delta+s=r\le h$ and thus $\Delta\ge n+1$, a contradiction.
\item
For $r\in [h]$ and $s \in \{k+1,\ldots,\Delta\}$, then we have that $\min\{\pi(r),\pi(k)\}=\pi(r) \ne 
\min\{\pi(k),\pi(n-\Delta+s)\}={\pi(n-\Delta+s)}$,
since the indices $\pi(r),\pi(n-\Delta +s)$ are distinct.
Indeed, equality $\pi(r)=\pi(n-\Delta+s)$ would imply that $r=n-\Delta+s$ and thus 
$2(n-\Delta)+h+1=n-\Delta +k+1 \le n-\Delta+s=r \le h$, which implies $\Delta \geq n+1$, a contradiction.
\item For $r\in \{h+1,\ldots,k\}$ and $s\in \{k+1,\ldots,\Delta\}$, 
then we have that $\min\{\pi(r),\pi(n-\Delta+r)\} \ne \min\{\pi(k),\pi(n-\Delta+s)\}=\pi(n-\Delta+s)$.
Indeed, $r=n-\Delta+s$ would imply that
$n-\Delta+k+1\le n-\Delta+s=r\le k$, which implies $\Delta \geq n+1$, a contradiction.
\end{itemize}

Thus (C2) holds, which concludes the proof in case (4) and thus the proof of the theorem.
\end{proof}

\section{Applications of the main result}\label{secapplication}

We now formulate several applications  of our main result in Theorem \ref{theomain1}.
As a first direct consequence, we can show that the identity matrix is an optimal solution to the problem $\QAP(A,B)$ whenever $A$ is a Robinson similarity matrix, $B$ is a  Robinson dissimilarity matrix, and at least one of $A$ or $B$ is a Toeplitz matrix.

\begin{theorem} \label{theomain2}
Let $A,B\in \MS^n$ and assume that $A$ is a Robinson similarity matrix, $B$ is a Robinson dissimilarity matrix and moreover $A$ or $B$ is a Toeplitz matrix.
Then the identity permutation is an optimal solution to the problem $\QAP(A,B)$.
\end{theorem}

\begin{proof}
Assume first that $B$ is a Toeplitz matrix. Then, by Lemma \ref{lemToeplitz}, $B$ is a conic combination of the matrices $B^\Delta_n$, i.e.,
$B=\sum_{\Delta=1}^{n-1} \lambda_\Delta B^\Delta_n$ for some scalars $\lambda_\Delta\ge 0$.
Applying the inequality (\ref{eq1}) in Theorem \ref{theomain1}, we obtain that 
$$\langle A_\pi,B\rangle =\sum_{\Delta=1}^{n-1}\lambda_\Delta \langle A_\pi,B^\Delta_n\rangle \ge 
\sum_{\Delta=1}^{n-1}\lambda_\Delta \langle  A,B^\Delta_n\rangle =\langle A,B\rangle,$$
which shows that the identity permutation is optimal for $\QAP(A,B)$.

Assume now that $A$ is a Toeplitz Robinson similarity (and $B$ is a Robinson dissimilarity). Then we simply exchange the roles of $A$ and $B$ after a simple modification.  Namely, let $\alpha$ and $\beta$ denote the maximum value of the entries of $A$ and $B$, respectively.
Let define the two matrices $B'= \alpha J -A$ and $A'=\beta J-B$. 
Then $A'$ is a Robinson similarity matrix and $B'$ is a Toeplitz  Robinson dissimilarity matrix.
For any permutation $\pi$, by the previous result applied to $\QAP(A',B')$, $\langle (A')_{\pi},(B') \rangle \ge \langle A',B'\rangle$.
If we compute the inner product of both sides, we obtain 
$\langle (A')_{\pi},B' \rangle = \alpha\beta \langle J_{\pi},J\rangle -\alpha \langle J_{\pi},A \rangle - \beta \langle B_{\pi},J\rangle + \langle B_{\pi}, A\rangle$ and 
$\langle A',B' \rangle = \alpha\beta \langle J,J\rangle -\alpha \langle J,A \rangle - \beta \langle B,J\rangle + \langle B, A\rangle$, from which we can easily conclude that $ \langle A , B_{\pi} \rangle \geq \langle  A, B\rangle$.
Hence, this shows that the identity permutation is optimal for $\QAP(A,B)$ and it concludes the proof.
\end{proof}

As a direct application,  Theorem \ref{theomain2} extends to the case when the matrices $A$ and $B$ are Robinsonian.

\begin{corollary}\label{theomain3}
Let $A,B\in \MS^n$. Assume that $A$ is a Robinsonian similarity matrix, $B$ is a Robinsonian dissimilarity matrix, and let $\pi$, $\tau$  be permutations that  reorder $A$ and $B$ as Robinson similarity and dissimilarity matrices, respectively.
Assume furthermore that one of the matrices $A_{\pi}$ or $B_{\tau}$ is a Toeplitz matrix.
Then the permutation $\tau^{-1}\pi$ is optimal for the problem $\QAP(A,B)$.
\end{corollary}

\begin{proof}
Directly from Theorem \ref{theomain2}, using relation (\ref{eqpi}).
\end{proof}

Finally we observe  that the assumption that either $A$ or $B$ has a Toeplitz structure cannot be omitted in Theorem \ref{theomain1}. Indeed, consider the following matrices:
\begin{equation*}
A =
\left( \begin{matrix}
1 & 1 & 1 & 0 & 0\\
1 & 1 & 1 & 1 & 0\\
1 & 1 & 1 & 1 & 0\\
0 & 1 & 1 & 1 & 0\\
0 & 0 & 0 & 0 & 1
\end{matrix}\right)
,\quad
B =
\left( \begin{matrix}
0 & 1 & 1 & 1 & 1\\
1 & 0 & 1 & 1 & 1\\
1 & 1 & 0 & 0 & 0\\
1 & 1 & 0 & 0 & 0\\
1 & 1 & 0 & 0 & 0
\end{matrix}\right) 
\end{equation*}
so that $A$ (resp., $B$) is a Robinson similarity (resp., dissimilarity). In this case, the identity permutation gives a solution of value 
 $\langle A,B \rangle= 8$. Consider now the permutation $\pi = (4,5,1,2,3)$ which reorders $A$ as follows:
\begin{equation*}
A_\pi =
\left( \begin{matrix}
1 & 0 & 0 & 1 & 1\\
0 & 1 & 0 & 0 & 0\\
0 & 0 & 1 & 1 & 1\\
1 & 0 & 1 & 1 & 1\\
1 & 0 & 1 & 1 & 1
\end{matrix}\right).
\end{equation*}
This gives a solution of value  $\langle A_\pi,B \rangle = 4$. Hence, in this case the identity permutation is not optimal, and thus the Toeplitz assumption cannot be removed in Theorem \ref{theomain1}.

To the best of our knowledge the complexity of $\QAP(A,B)$, when $A$ is a Robinson similarity and $B$ is a Robinson dissimilarity, is not known.

\section*{Acknowledgements}
This work was supported by the Marie Curie Initial Training Network ``Mixed Integer Nonlinear Optimization" (MINO).
We thank two anonymous referees for their useful comments. In particular we are grateful to a referee for pointing out to us  related results in  the references \cite{Christopher96} and \cite{Deineko98}, and 
to Gerhard Woeginger for pointing out to us the paper 
\cite{Prea14}.

\end{document}